\documentclass[a4paper,oneside,10pt]{article}%
\usepackage{amsmath}
\usepackage{amsfonts}
\usepackage{amssymb}
\usepackage{authblk}
\usepackage{graphicx}
\usepackage{subfigure}
\usepackage{color}
\usepackage{caption,subfigure}
\usepackage[square,numbers,sort&compress]{natbib}%
\usepackage{float}
\usepackage{algorithm}
\usepackage{algorithmic}
\usepackage{multirow}
\usepackage{geometry}
\usepackage{graphics}
\usepackage{url}
\usepackage{mathptmx}
\providecommand{\U}[1]{\protect\rule{.1in}{.1in}}

\newtheorem{theorem}{Theorem}[section]

\newtheorem{assumption}[theorem]{Assumption}
\newtheorem{example}[theorem]{Example}

\newtheorem{lemma}[theorem]{Lemma}

\newtheorem{remark}[theorem]{Remark}

\newenvironment{proof}[1][Proof]{\noindent\textbf{#1.} }{\ \rule{0.5em}{0.5em}}
\numberwithin{equation}{section}

\allowdisplaybreaks[4]

\title{Nonparametric estimation of FBSDEs with random terminal time}
\author{Shaolin Ji\thanks{ Email: jsl@sdu.edu.cn }\quad
Chenyao Yu\thanks{ Email: yucy@mail.sdu.edu.cn }\quad
  Linlin Zhu\thanks{Corresponding author. Email: 201611343@mail.sdu.edu.cn } }

\affil{Institute for Financial Studies,
Shandong University, Jinan, Shandong 250100, PR China}

\begin{document}

\maketitle
\textbf{Abstract}.  This paper investigates the nonparametric estimation of the functional coefficients of the FBSDEs with random terminal time, including the local constant and local linear estimators. We provide complete two-dimensional asymptotics in both the time span and the sampling interval, allowing for the precise characterization of their distribution. Moreover, the empirical likelihood (EL) method to construct the data-driven confidence intervals for these estimators is provided.
Some numerical simulations investigate the finite-sample properties of the estimators and compare the performance of the EL method and the conventional method in constructing confidence intervals based on asymptotic normality.



\addcontentsline{toc}{section}{\hspace*{1.8em}Abstract}

\section{ Introduction}
Nonlinear BSDEs were introduced by
Pardoux and Peng \cite{peng1990}. It was shown in \cite{peng1991} that coupled with a forward SDE, such BSDEs give a probabilistic interpretation for systems of quasilinear parabolic and elliptic partial differential equations (PDEs), which generalized the classical Feynman-Kac formula for linear parabolic and elliptic PDEs. Later, many researchers developed the theory of FBSDEs and their applications. In mathematical finance, FBSDEs can be used in the theory of stochastic differential utility and the theory of contingent claim evaluation for a large investor. The existence and uniqueness of the solutions to some kinds of FBSDEs are closely linked to optimal stochastic control problems. In many applications, there is a lack of prior information about the structure of the model. Therefore, it becomes important to identify and estimate the parameters and functionals of the process with discretely observed data. In the past few decades, there has been a thorough study on the nonparametric estimation of the (jump-) diffusion processes and stochastic volatility models (see \cite{ait2016}, \cite{bandi2003}, \cite{bandi2018}, \cite{bandisv2018}, \cite{fan2003}, \cite{Florens1993}, \cite{Jacod2000}, \cite{Stanton1997}, \cite{wang2018}, \cite{xukeliEL}, \cite{xukeli}). However, as one of the important tools used to describe the stochastic process in many fields, research on the statistical inference of BSDEs is rarely performed.
In this paper, we consider the nonparametric estimation of  the forward-backward stochastic differential equations (FBSDEs) with random terminal time.

Consider the following FBSDE,
\begin{equation}
\label{model-FBSDE-general}
\left\{\begin{array}{l}
X_{t}=x+\int_{0}^{t} b\left(X_{s}\right) \mathrm{d} s+\int_{0}^{t} \sigma\left(X_{s}\right) \mathrm{d} W_{s}, \\
Y_{t}=g\left(X_{\tau}\right)+\int_{t\land \tau}^{\tau} f\left(X_{s}, Y_{s}, Z_{s}\right) \mathrm{d} s-\int_{t\land \tau}^{\tau} Z_{s} \mathrm{d}W_{s},
\end{array}\right.
\end{equation}
where $\left\{W_{t}\right\}_{t \geq 0}$ is a $d$-dimensional standard Brownian motion defined on a complete filtered probability space $(\Omega, \mathcal{F}, P)$. $\{\mathcal{F}_{t}\}_{t\geq 0}$ is the natural filtration of this Brownian motion such that $\mathcal{F}_{0}$ contains all $P$-null elements of $\mathcal{F}$. $\tau$ is an $\mathcal{F}_{t}$-stopping time with values in $[0,\infty]$.
$X, Y, Z$ take values in $\mathbb{R}^{d}, \mathbb{R}^{n}, \mathbb{R}^{n \times d}$, respectively.
Assume the specific form of the functional coefficients $b$, $\sigma$, and $f$ is unknown, but only the process $X_{t}$ and $Y_{t}$ are observed at discrete times. Based on these observations, we are interested in establishing the point-wise estimators of the functional coefficients and the term $Z$.

There are only a few studies on the statistical inference of BSDEs.
Su and Lin \cite{sulin2009} proposed the local constant estimator of $Z$ and the least square estimator of the unknown parameter in the generator for a linear FBSDE with fixed terminal time. Chen and Lin \cite{lin2010} considered a coupled Markovian FBSDE with fixed terminal time. They constructed the local linear estimators of the functional coefficients and $Z$ and obtained the asymptotic properties of these estimators.
In these articles, some assumptions are strict, such as some stationary assumptions. Moreover, their estimators are based on the relationship between FBSDEs with deterministic terminal time and quasi-linear parabolic PDEs, in which the solutions $\{Y_{t},Z_{t}\}$  and the functional coefficients of the FBSDEs are all deterministic functions of $X_{t}$ and $t$. However, there is insufficient information to estimate the bivariate functions with only one observed trajectory.
It is well known that BSDEs with random terminal time are connected with quasi-linear elliptic PDEs. The solutions to these BSDEs are deterministic functions of $\{X_{t}\}$, i.e., they depend on the current state of $X_{t}$, not on time. Applying this property, we construct the local constant estimators and local linear estimators of the functional coefficients and $Z$ for FBSDEs with random terminal time.
The asymptotic properties of the estimators are obtained under mild conditions (the process $X_{t}$ and $Y_{t}$ need not be stationary).
The local polynomial estimator, the double-smoothing estimator, can be easily generalized.

Given these estimators, the common confidence intervals can be obtained based on their asymptotic normality, provided consistent estimators of the asymptotic variances are available. However, this type of confidence interval is always symmetric, and the estimated variances usually have large biases. We apply the empirical likelihood in conjunction with the local constant estimators to construct the data-driven confidence intervals that can account for possible skewness of the estimators and avoid imprecise variance estimation. We focus on local constant smoothing, and the extensions to the local linear or local polynomial estimators are straightforward.

This paper is organized as follows. The models and some preliminaries are introduced in Section $2$. The estimators and their asymptotic properties are presented in Section $3$. 
Section $4$ introduces the EL method to construct the point-wise confidence interval. The proofs are given in Section $5$.
Some numerical examples are shown in Section $6$.

\section{Model setup and preliminaries}
Denote
$\mathbb{L}^{2,k}(0,\tau;\mathbb{R}^{n})$, where $k\in \mathbb{R}$, the space of $\mathbb{R}^{n}$-valued progressively measurable processes $f$ such that
\begin{flalign*}
\|f(\cdot)\|_{\mathbb{L}^{2,k}}:=\left(\mathbb{E}\int_{0}^{\tau}|f(\cdot)|^{2}e^{kt}dt\right)^{1/2} \leq \infty,
\end{flalign*}
and                                                                                                                                                                                                      $\mathbb{L}^{2}(0,\tau;\mathbb{R}^{n}):=\mathbb{L}^{2,0}(0,\tau;\mathbb{R}^{n})$;
Further, define
\begin{flalign*}
\mathcal{L}^{2,k}(0,\tau):=\mathbb{L}^{2,k}(0,\tau;\mathbb{R}^{n}) \times \mathbb{L}^{2,k}(0,\tau;\mathbb{R}^{n\times d}),\ \ \ \mathcal{R}:=\mathbb{R}^{d}\times\mathbb{R}^{n}\times\mathbb{R}^{n\times d}.
\end{flalign*}
Let $\{X_{t}^{x};t\geq 0\}$ denote the solution of the following forward SDE:
\begin{flalign}
\label{model-SDE}
X_{t}^{x}=x+\int_{0}^{t}b\left(X_{s}^{x}\right) ds+\int_{0}^{t}\sigma\left(X_{s}^{x}\right) dW_{s},\ t\geq 0,
\end{flalign}
where $x \in \mathbb{R}^{d}$, $b:\mathbb{R}^{d} \rightarrow \mathbb{R}^{d}$, $\sigma:\mathbb{R}^{d} \rightarrow \mathbb{R}^{d\times d}$ are globally Lipschitz and twice continuously differentiable.

\subsection{BSDEs with infinite horizon}
Consider the infinite horizon BSDE which is defined as follows,
\begin{flalign}
\label{model-FBSDE-infinite}
Y_{t}^{x}=Y_{T}^{x}+\int_{t}^{T}f\left(X_{s}^{x},Y_{s}^{x},Z_{s}^{x}\right) ds-\int_{t}^{T}Z_{s}^{x} dW_{s},\ \forall t,T\ s.t.\ 0\leq t\leq T,
\end{flalign}
where $f: \mathcal{R}\to\mathbb{R}^{n}$ is a continuous function, the process $\{X_{t}^{x};t\geq 0\}$ is defined as $(\ref{model-SDE})$.

\begin{assumption}
\label{assumption-BSDE-infinite-solution}
There exists some constants  $C_{1}> 0,\ C_{2}> 0,\ \mu < 0,\ p >0$ such that
\begin{flalign*}
&
|f(x,y,z)| \leq C_{1}\left(1+|x|^{p}+|y|+\|z\|\right),\\
&\left<f(x,y_{1},z)-f(x,y_{2},z),y_{1}-y_{2}\right> \leq \mu|y_{1}-y_{2}|^{2},\\
&|f(x,y,z_{1})-f(x,y,z_{2})|\leq C_{2}\|z_{1}-z_{2}\|.
\end{flalign*}
Moreover, for some $\lambda > 2\mu+C_{2}^{2}$, and all $x\in \mathbb{R}^{d}$,
\begin{flalign*}
&\mathbb{E}\int_{0}^{\infty}|f(X_{t}^{x},0,0)|^{2}e^{\lambda t}dt < \infty.
\end{flalign*}
\end{assumption}

According to Theorem $3.1$ in \cite{pardoux1998},
under Assumption \ref{assumption-BSDE-infinite-solution}, the BSDE $(\ref{model-FBSDE-infinite})$ has a unique solution $\{(Y_{t},Z_{t});t\geq 0\}$ which belongs to $\mathcal{L}^{2,\lambda}(0,\infty)$.
Consider a semilinear elliptic PDE in $\mathbb{R}^{d}$ which is of the form
\begin{flalign}
\label{model-pde-infinite horizon}
\mathbb{L} u(x)+f\left(x, u(x),(\nabla u \sigma)(x)\right)=0,\ x\in \mathbb{R}^{d},
\end{flalign}
where
\begin{flalign*}
\mathbb{L}=\sum_{i}^{d} b_{i}(x)\frac{\partial}{\partial x_{i}} +\frac{1}{2}\sum_{ij}^{d} (\sigma\sigma^{*})_{ij}(x) \frac{\partial^{2}}{\partial x_{i}\partial x_{j}}
\end{flalign*}
is the infinitesimal generator of the Markov process $\{X_{t}^{x};t\geq 0\}$.

\begin{lemma}[\cite{pardoux1998}, Theorem $4.1$]
Let Assumption $\ref{assumption-BSDE-infinite-solution}$ hold.
If $u\in C^{2}(\mathbb{R}^{d};\mathbb{R}^{n})$ is a classical solution of $(\ref{model-pde-infinite horizon})$ such that
\begin{flalign*}
\mathbb{E}\left[\int_{0}^{\infty}e^{\lambda t}\|
(\nabla u\sigma)(X_{t}^{x})\|^{2}dt\right]<\infty,\ x\in \mathbb{R}^{d}.
\end{flalign*}
Then for each $x\in \mathbb{R}^{d}$, $\{u(X_{t}^{x}),(\nabla u \sigma) (X_{t}^{x}); t>0\} $ is the unique solution of the  BSDE
$(\ref{model-FBSDE-infinite})$.
\end{lemma}

\subsection{BSDEs with finite random terminal time}
Let $G$ be an open bounded subset of $\mathbb{R}^{d}$ with boundary of class $C^{1}$. For each $x\in \bar{G}$, define
the stopping time
$\tau_{x} \equiv \inf \left\{t \geq 0: X^{x}_{t} \notin \bar{G}\right\}$.
Assume that $P(\tau_{x}<\infty)=1$, and for all $x\in \bar{G}$, the set $\Gamma=\{x \in \partial G; P\left(\tau_{x}>0\right)=0\}$ is closed.

Consider the following BSDE
\begin{flalign}
\label{model-FBSDE with stopping time}
Y^{x}_{t}=g\left(X^{x}_{\tau_{x}}\right)+\int_{t \wedge \tau_{x}}^{\tau_{x}} f\left(X_{s}^{x},Y_{s}^{x},Z_{s}^{x}\right) d s-\int_{t \wedge \tau_{x}}^{\tau_{x}} Z^{x}_{s} dW_s,
\end{flalign}
where $g\in C(\mathbb{R}^{d})$, $f$ is continuous and satisfies Assumption $\ref{assumption-BSDE-infinite-solution}$. Assume that for some $\lambda>2\mu+C_{2}^{2}$, $\mathbb{E}\left[\exp \left(\lambda \tau_{x}\right)\right]<\infty$.
Then, the BSDE $(\ref{model-FBSDE with stopping time})$ has a unique solution $\{(Y_{t},Z_{t});t\geq 0\}$ in $\mathcal{L}^{2,\lambda}(0,\tau_{x})$.
Consider the PDE with Dirichlet boundary condition which is of the form
\begin{flalign}
\label{model-PDE-exit time}
\left\{\begin{array}{l}
\mathbb{L} u(x)+f(x, u(x),(\nabla u\sigma)(x))=0, \quad x \in G, \\
u(x)=g(x), \quad x \in \partial G.
\end{array}\right.
\end{flalign}

\begin{lemma}
Let Assumption \ref{assumption-BSDE-infinite-solution} hold.
If $(\ref{model-PDE-exit time})$ has a classical solution $u\in C^{2}(\bar{G};\mathbb{R}^{n})$,
then  $\{u(X_{t}^{x}),(\nabla u\sigma)(X_{t}^{x}); 0\leq t\leq \tau_{x}\} $ is the unique solution of the BSDE $(\ref{model-FBSDE with stopping time})$.
\end{lemma}

\begin{remark}
The conditions to guarantee a unique solution $u\in C^{2}$ of the elliptic equation $(\ref{model-PDE-exit time})$  can be referred to Theorem 4.1 in \cite{peng1991}.
\end{remark}

\section{Nonparametric estimation of the coefficients}
In this section, we will construct the nonparametric estimators of $f(X_{t},Y_{t},Z_{t})$ and $Z^{2}_{t}$ in the models $(\ref{model-FBSDE-infinite})$ and $(\ref{model-FBSDE with stopping time})$. We consider the one-dimensional case for simplicity, and similar results can be generalized in multi-dimensional cases under a more complex proof.
Assume $\mathcal{D}=(\bar{l},\bar{r})$ is the range of the process $X_{t}$ and $Y_{t}$ in the two models, where $-\infty \leq \bar{l} \leq \bar{r} \leq \infty$.

In the following, assume that $(\ref{model-pde-infinite horizon})$ and $(\ref{model-PDE-exit time})$ both have a unique solution $u \in  C^{2}$ and  Assumption $\ref{assumption-BSDE-infinite-solution}$ holds.
Then, the solutions of the above BSDEs could be represented as deterministic functions of $X_{t}$. With this conclusion, we denote $-f(X_{t},Y_{t},Z_{t})$ and $Z_{t}$ as  $\tilde{f}(X_{t})$ and $\tilde{Z}(X_{t})$, respectively.
Note that for the diffusion process $\{X_{t}\}$ satisfying $(\ref{model-SDE})$ and a smooth function $Q$, the conditional expected increment can be expressed as
\begin{flalign}
\label{conditional expectation expansion}
\mathbb{E}\left[Q\left(X_{t+\Delta}\right)-Q\left(X_{t}\right) \mid X_{t}\right]=\mathbb{L} Q\left(X_{t}\right) \Delta+\frac{1}{2} \mathbb{L}^{2} Q\left(X_{t}\right) \Delta^{2}+O\left(\Delta^{3}\right).
\end{flalign}
Setting $Q(x)=u(x)-u(X_{t})$ and $\left(u(x)-u(X_{t})\right)^{2}$ in $(\ref{conditional expectation expansion})$, we have
\begin{flalign}
\label{g-expection}
\tilde{f}(X_{t})=\lim\limits_{\Delta \rightarrow 0}\frac{1}{\Delta} \mathbb{E}\left[ (Y_{t+\Delta}-Y_{t}) | X_{t}\right],
\end{flalign}
and
\begin{flalign}
\label{z-expection}
\tilde{Z}^{2}(X_{t})=\lim\limits_{\Delta \rightarrow 0}\frac{1}{\Delta} \mathbb{E}\left[(Y_{t+\Delta}-Y_{t})^{2}| X_{t}\right],
\end{flalign}
respectively for both $(\ref{model-FBSDE-infinite})$ and $(\ref{model-FBSDE with stopping time})$.
Assume the processes $X_{t}$ and $Y_{t}$ in $(\ref{model-FBSDE-infinite})$ and $(\ref{model-FBSDE with stopping time})$ are observed at $\{t=t_{0},t_{1},t_{2},\dots,t_{n}\}$ in the time interval $[0,T]$ with $t_{0}=0$ and $t_{n}=T$, and the observations are equispaced. Then, \\
$\left\{X_{0}, X_{\Delta}, \dots, \allowbreak X_{(n-1)\Delta}, X_{n\Delta}\right\}$, $\left\{Y_{0},Y_{\Delta}, \dots, Y_{(n-1)\Delta}, Y_{n\Delta}\right\}$ are the observations, where $\Delta=T/n$.

Based on the infinitesimal conditional moment restriction (\ref{g-expection}) and (\ref{z-expection}), the common nonparametric regression methods, such as local constant regression, local polynomial regression can be exploited to estimate $\tilde{f}(x)$ and $\tilde{Z}^{2}(x)$ at every spatial point $x$.

\subsection{Nonparametric kernel estimation of the generator $\tilde{f}(x)$}
\subsubsection{Locally constant estimator}
The locally constant estimator of $\tilde{f}(x)$ for $(\ref{model-FBSDE-infinite})$ and $(\ref{model-FBSDE with stopping time})$ is defined as follows,
\begin{flalign}
&\label{g-kernelestimator}
\hat{\tilde{f}}_{\text{NW}}(x):= \frac{N(K,\tilde{f}) }
{D(K)},
\end{flalign}
where $N(K,\tilde{f})=\frac{1}{\Delta}
\sum_{i=0}^{n-1} K\left(\frac{X_{i\Delta}-x}{h}\right)
\left(Y_{(i+1)\Delta}-Y_{i\Delta}\right)$, $D(K)=\sum_{i=0}^{n-1} K\left(\frac{X_{i\Delta}-x}{h}\right)$.
Let
$$\hat{\tilde{f}}_{\text{NW}}(x)=\hat{g}_{p,\text{NW}}(x)+\hat{g}_{q,\text{NW}}(x)+\hat{g}_{r,\text{NW}}(x),$$
where
\begin{equation*}
\hat{g}_{p,\text{NW}}(x)=\tilde{f}(x)+\frac{B(K,\tilde{f})}{D(K)},\ \
\hat{g}_{q,\text{NW}}(x)=\frac{M(K,\tilde{f})}{D(K)},\ \
\hat{g}_{r,\text{NW}}(x)=\frac{R(K,\tilde{f})}{D(K)},
\end{equation*}
with
\begin{flalign*}
& B(K,\tilde{f})=\frac{\Delta}{h} \sum_{i=0}^{n-1} K\left(\frac{X_{i \Delta}-x}{h}\right) \left(\tilde{f}\left(X_{i\Delta}\right)-\tilde{f}(x)\right),    \\
& M(K,\tilde{f})=\frac{1}{h} \sum_{i=0}^{n-1} K\left(\frac{X_{i\Delta}-x}{h}\right)
\int_{i\Delta}^{(i+1)\Delta}Z_{s}dW_{s},    \\
&R(K,\tilde{f})=\frac{1}{h} \sum_{i=0}^{n-1} K\left(\frac{X_{i\Delta}-x}{h}\right)
\int_{i\Delta}^{(i+1)\Delta}\left(\tilde{f}(X_{s})-\tilde{f}(X_{i\Delta})\right)ds.
\end{flalign*}
To give the asymptotic property of the estimator, we make the following assumptions.
\begin{assumption}
\label{assum-kernel function}
The kernel function $K(\cdot)$ is a bounded, twice continuously differentiable, symmetric function with a compact support and for which
$\int_{-\infty}^{\infty}K(u)du=1$, $\int_{-\infty}^{\infty}uK(u)du=0$.
Define $l(K_{1})=\int_{-\infty}^{\infty}u^{2}K(u)du$, $l(K_{2})=\int_{-\infty}^{\infty}K^{2}(u)du$.
\end{assumption}
These conditions can be satisfied by many kernel functions, including the Epanechnikov kernel $K(u)=\frac{3}{4}(1-u^{2})I_{\{|u|\leq 1\}}$, the Quartic kernel $K(u)=\frac{15}{16}(1-u^{2})^{2}I_{\{|u|\leq 1\}}$, where $I$ is the indicator function.

\begin{assumption}
\label{assum-recurrence}
The process $X_{t}$  is recurrent, i.e., the scale function of $X_{t}$
\begin{flalign*}
S(\alpha)=\int_{c}^{\alpha}exp\left(\int_{c}^{y}\frac{-2b(x)}{\sigma^{2}(x)}dx\right)dy,
\end{flalign*}
where $c$ is a generic fixed number belonging to $\mathcal{D}$, satisfies
\begin{flalign*}
\lim\limits_{\alpha\rightarrow \bar{l}}S(\alpha)=-\infty,\ \ \lim\limits_{\alpha\rightarrow \bar{r}}S(\alpha)=\infty.
\end{flalign*}
\end{assumption}
Intuitively, recurrence ensures that the Markov process $X_{t}$ could visit every spatial point $x \in \mathcal{D}$ infinite times with probability one when $T \rightarrow \infty$. The recurrence-related concepts and their applications for statistical inference of stochastic processes can be found in \cite{ait2016}, \cite{bandi2003}. This condition does not imply the existence of a time-invariant distribution for $X_{t}$, therefore, nonstationary is allowed.

To illustrate the asymptotic properties of the estimators, we apply the local time $\ell(T,x)$ of the Markov process $X_{t}$ defined in \cite{ait2016}. It measures the amount of calendar time spent by the process in the neighborhood of $x$. For a recurrent diffusion process, $\ell(T,x)$ diverges as $T\rightarrow \infty$ at every spatial point $x$.
In this paper, $\ell(T,x)$ of the Markov process $X_{t}$ in $(\ref{model-FBSDE-infinite})$ and $(\ref{model-FBSDE with stopping time})$ need to satisfy the following Assumption.

\begin{assumption} 
\label{assum-local time}
$(i)$ There exits $\varepsilon >0$ such that $\bar{\ell}_{\varepsilon}(T, x)=O_{p}\left(\ell(T, x)^{2}\right)$, where $\bar{\ell}_{\varepsilon}(T,x)=\sup_{|x-y|\leq\varepsilon}\ell(T,y)$.   \\
$(ii)$ $\ell(T,x)=O_{p}(a_{T})$ for some nonrandom sequence $(a_{T})$.
\end{assumption}

\begin{assumption}
\label{assumf}
Given $\Delta \rightarrow 0$ and $h \rightarrow 0$ such that    \\
(i) $\Delta^{1/8}/h^{2}=o_{p}(1)$. $\Delta^{1/8}(\log n)^{1/2} =o_{p}(1)$.
$\Delta^{1/8}T_{\bar{H}_{1}} =o_{p}(1)$, where $\bar{H}_{1}=\sigma, Z, Z', \tilde{f}'$.    \\
(ii) $\Delta^{3/4}T_{\bar{H}_{2}} =o_{p}(1)$, where $\bar{H}_{2}=\tilde{f}, b$.
\end{assumption}

\begin{theorem}
\label{thm-g-NW}
Let Assumption $\ref{assum-kernel function}$ - $\ref{assumf}$ hold, we have
\begin{flalign*}
\hat{g}_{p,\text{NW}}(x)=\tilde{f}(x)+h^{2}l(K_{1})\tilde{f}'(x) \frac{s'(x)}{s(x)}+o_{p}(h)+O_{p}(h^{1/2}\ell(T,x)^{-1/2})
\end{flalign*}
uniformly in $T$ as $h\to 0$ and $\Delta\to 0$,  where $s(x)$ is the speed function of the process $X_{t}$ at the generic lever $x$, given by $s(x)=2/S^{'}(x)\sigma^{2}(x)$. And,
\begin{flalign*}
&(h \ell(T,x))^{1/2}\hat{g}_{q,\text{NW}}(x)\Rightarrow l(K_{2})^{1/2}\tilde{Z}(x)\mathbb{N},
\end{flalign*}
where $\mathbb{N}$  is a standard normal random variable. Furthermore,
$\hat{g}_{r,\text{NW}}(x)=o_{p}(h^{2})$.
\end{theorem}

\subsubsection{Locally linear estimator}
The locally linear estimator of $\tilde{f}(x)$ for $(\ref{model-FBSDE-infinite})$ and $(\ref{model-FBSDE with stopping time})$ is given by,
\begin{flalign}
&\label{g-kernelestimatorll}
\hat{\tilde{f}}_{\text{LL}}(x):= \frac{N(K,\tilde{f})D(K_{2})-N(K_{1},\tilde{f})D(K_{1})}
{D(K)D(K_{2})-D(K_{1})^2},
\end{flalign}
where $N(K_{1},\tilde{f})$  is defined as $N(K,\tilde{f})$,  $D(K_{1})$ and $D(K_{2})$ as $D(K)$, only with $K$ replaced by $K_1$ and $K_2$, $K_1(x) =xK(x)$ and $K_2(x)= x^2K(x)$. Write
$$\hat{\tilde{f}}_{\text{LL}}(x)=\hat{g}_{p,\text{LL}}(x)+\hat{g}_{q,\text{LL}}(x)+\hat{g}_{r,\text{LL}}(x),$$
where
\begin{flalign*}
&\hat{g}_{p,\text{LL}}(x)=\tilde{f}(x)+\frac{B(K,\tilde{f})D(K_{2})-B(K_{1},\tilde{f})D(K_{1})}{D(K)D(K_{2})-D(K_{1})^2},\\
&\hat{g}_{q,\text{LL}}(x)=\frac{M(K,\tilde{f})D(K_{2})-M(K_{1},\tilde{f})D(K_{1})}{D(K)D(K_{2})-D(K_{1})^2}, \\
&\hat{g}_{r,\text{LL}}(x)=\frac{R(K,\tilde{f})D(K_{2})-R(K_{1},\tilde{f})D(K_{1})}{D(K)D(K_{2})-D(K_{1})^2}.
\end{flalign*}

\begin{theorem}
\label{thm-g-LL}
Let Assumption $\ref{assum-kernel function}$ - $\ref{assumf}$ hold, we have
\begin{flalign*}
\hat{g}_{p,\text{LL}}(x)=\tilde{f}(x)+o_{p}(h)+O_{p}(h^{1/2}\ell(T,x)^{-1/2})
\end{flalign*}
uniformly in $T$ as $h\to 0$ and $\Delta\to 0$, and,
\begin{flalign*}
&~~(h \ell(T,x))^{1/2}\hat{g}_{q,\text{LL}}(x)\Rightarrow l(K_{2})^{1/2}\tilde{Z}(x)\mathbb{N}.
\end{flalign*}
Furthermore,
$\hat{g}_{r,\text{LL}}(x)=o_{p}(h^{2})$.
\end{theorem}
\begin{remark}
Note that  $\hat{\tilde{f}}_{\text{NW}}(x)$ and $\hat{\tilde{f}}_{\text{LL}}(x)$  converge only if
$h\ell(T,x) \stackrel{p}{\longrightarrow} \infty$. As we all know, for a recurrent process, we have $\ell(T,x) \to \infty$ as $T \to \infty$ at each $x\in \mathcal{D}$, $\ell(T,x)$ does not diverge for the transient process. Therefore, if the process $X_{t}$ is recurrent, the local estimator of the generator $f(x,u(x),v(x))$ for the infinite horizon BSDE $(\ref{model-FBSDE-infinite})$ is consistent. This does not apply to the BSDE $(\ref{model-FBSDE with stopping time})$.
\end{remark}

\subsection{Nonparametric kernel estimation of $\tilde{Z}^{2}(x)$}
\subsubsection{Locally constant estimator}
The locally constant estimator for $\tilde{Z}^{2}(x)$ is given by
\begin{flalign}
&\label{z-kernelestimatorll}
\hat{\tilde{Z}}^{2}_{\text{NW}}(x):=\frac{N(K,Z)} {D(K) },
\end{flalign}
where
\begin{flalign*}
N(K,Z)=\frac{1}{\Delta}
\sum_{i=0}^{n-1} K\left(\frac{X_{i\Delta}-x}{h}\right) \left(Y_{(i+1)\Delta}-Y_{i\Delta} \right)^{2}.
\end{flalign*}
Write
\begin{flalign*}
\hat{\tilde{Z}}^{2}_{\text{NW}}(x)=\hat{Z}^{2}_{p,\text{NW}}(x)+\hat{Z}^{2}_{q,\text{NW}}(x)+\hat{Z}^{2}_{r,\text{NW}}(x),
\end{flalign*}
where
\begin{flalign*}
\hat{Z}^{2}_{p,\text{NW}}(x)=\tilde{Z}^{2}(x)+\frac{ B(K,Z)  }{D(K)},\quad
\hat{Z}^{2}_{q,\text{NW}}(x)=\frac{M(K,Z)  }{D(K)},\quad
\hat{Z}^{2}_{r,\text{NW}}(x)=\frac{R(K,Z)  }{D(K)}+\frac{ S(K)  }{D(K)},
\end{flalign*}
with
\begin{flalign*}
&B(K,Z)=\frac{\Delta}{h} \sum_{i=0}^{n-1} K\left(\frac{X_{i\Delta}-x}{h}\right) \left(\tilde{Z}^{2}(X_{i\Delta})-\tilde{Z}^{2}(x)\right),   \\
&M(K,Z)=\frac{1}{h} \sum_{i=0}^{n-1} K\left(\frac{X_{i\Delta}-x}{h}\right) \int_{i\Delta}^{(i+1)\Delta}2(Y_{t}-Y_{i\Delta})Z_{t}dW_{t} ,   \\
&R(K,Z)=\frac{1}{h} \sum_{i=0}^{n-1} K\left(\frac{X_{i\Delta}-x}{h}\right) \int_{i\Delta}^{(i+1)\Delta}\left(\tilde{Z}^{2}(X_{t})-\tilde{Z}^{2}(X_{i\Delta})\right)dt, \\
&S(K)=\frac{1}{h} \sum_{i=0}^{n-1} K\left(\frac{X_{i\Delta}-x}{h}\right) \int_{i\Delta}^{(i+1)\Delta}2(Y_{t}-Y_{i\Delta})\tilde{f}(X_{t})dt.
\end{flalign*}

\begin{assumption}
\label{assumz}
Given $\Delta \rightarrow 0$ and $h \rightarrow 0$ such that    \\
(i) $\Delta^{1/8}/h^2=o_{p}(1)$. $\Delta^{1/8}(\log n)^{1/2}=o_{p}(1)$.
$\Delta^{1/8}T_{H_{1}} =o_{p}(1)$, where $H_{1}=\sigma, \tilde{Z}, \tilde{Z}', (\tilde{Z}^2)'$.    \\
(ii) $\Delta^{1/4}T_{H_{2}} =o_{p}(1)$, where $H_{2}=\tilde{f}, |\tilde{f}|'$.    \\
(iii) $\Delta^{1/2}T_{b} =o_{p}(1)$.
\end{assumption}
\begin{theorem}
\label{thm-z-NW}
Let  Assumption $\ref{assum-kernel function}$ - $\ref{assum-local time}$, $\ref{assumz}$ hold. We have
\begin{flalign}
\label{proofthmzp}
\hat{Z}^{2}_{p,\text{NW}}(x)=\tilde{Z}^{2}(x)+h^{2}K_{1}(\tilde{Z}^{2})'(x) \frac{s'(x)}{s(x)}+o_{p}(h)+O_{p}(h^{1/2}\ell(T,x)^{-1/2})
\end{flalign}
uniformly in $T$ as $\Delta \rightarrow 0$ and $h \rightarrow 0$,  and,
\begin{flalign*}
\left(\frac{h\ell(T,x)}{\Delta}\right)^{1/2}\hat{Z}^{2}_{q,\text{NW}}(x)\Rightarrow \sqrt{2}l(K_{2})^{1/2}\tilde{Z}^{2}(x)\mathbb{N}.
\end{flalign*}
Moreover,
$\hat{Z}^{2}_{r,\text{NW}}(x)=o_{p}(h^{2})$.
\end{theorem}
\subsubsection{Locally linear estimator}
The locally linear estimator of $\tilde{Z}(x)$ for $(\ref{model-FBSDE-infinite})$ and $(\ref{model-FBSDE with stopping time})$ is defined as follows,
\begin{flalign}
&\label{g-kernelestimatorll}
~~\hat{\tilde{Z}}_{\text{LL}}(x):= \frac{N(K,\tilde{Z})D(K_{2})-N(K_{1},\tilde{Z})D(K_{1})}
{D(K)D(K_{2})-D(K_{1})^2},
\end{flalign}
where $N(K_{1},\tilde{Z})$  is defined as $N(K,\tilde{Z})$ with $K$ replaced by $K_1$. Write
$$\hat{\tilde{Z}}^{2}_{\text{LL}}(x)
=\hat{Z}^{2}_{p,\text{LL}}(x)+\hat{Z}^{2}_{q,\text{LL}}(x)+\hat{Z}^{2}_{r,\text{LL}}(x),$$
where
\begin{flalign*}
&\hat{Z}^{2}_{p,\text{LL}}(x)=\tilde{Z}^{2}(x)+\frac{B(K,\tilde{Z})D(K_{2})-B(K_{1},\tilde{Z})D(K_{1})}{D(K)D(K_{2})-D(K_{1})^2},\\
&\hat{Z}^{2}_{q,\text{LL}}(x)=\frac{M(K,\tilde{Z})D(K_{2})-M(K_{1},\tilde{Z})D(K_{1})}{D(K)D(K_{2})-D(K_{1})^2}, \\
&\hat{Z}^{2}_{r,\text{LL}}(x)=\frac{R(K,\tilde{Z})D(K_{2})-R(K_{1},\tilde{Z})D(K_{1})}{D(K)D(K_{2})-D(K_{1})^2}.
\end{flalign*}

\begin{theorem}
\label{thm-z-LL}
Let  Assumption $\ref{assum-kernel function}$ - $\ref{assum-local time}$, $\ref{assumz}$ hold. We have
\begin{flalign}
\hat{Z}^{2}_{p,\text{LL}}(x)=\tilde{Z}^{2}(x)+o_{p}(h)+O_{p}(h^{1/2}\ell(T,x)^{-1/2})
\end{flalign}
uniformly in $T$ as $\Delta \rightarrow 0$ and $h \rightarrow 0$,  and,
\begin{flalign*}
\left(\frac{h\ell(T,x)}{\Delta}\right)^{1/2}\hat{Z}_{\text{LL}}^{2}(x)\Rightarrow \sqrt{2}l(K_{2})^{1/2}\tilde{Z}^{2}(x)\mathbb{N}.
\end{flalign*}
Moreover,
$\hat{Z}^{2}_{r,\text{LL}}(x)=o_{p}(h^{2})$.
\end{theorem}

\begin{remark}
To ensure the consistency of $\hat{\tilde{Z}}^{2}_{\text{NW}}(x)$ and $\hat{\tilde{Z}}^{2}_{\text{LL}}(x)$, we only need $\Delta \to 0$ and $h \to 0$, we do not need $\ell(T,x)\stackrel{p}{\longrightarrow} \infty$ or $T\to \infty$. i.e., $\hat{\tilde{Z}}^{2}_{\text{NW}}(x)$ and $\hat{\tilde{Z}}^{2}_{\text{LL}}(x)$ could be consistent as long as $\Delta$ tends to $0$ sufficiently fastly relative to $h$, and the process $X_{t}$ need not be recurrent. So these estimators can be applied to both the infinite horizon BSDE $(\ref{model-FBSDE-infinite})$ and the BSDE $(\ref{model-FBSDE with stopping time})$.
\end{remark}

\section{Empirical likelihood point-wise confidence intervals}
This section investigates the empirical likelihood method in constructing the data-driven point-wise confidence intervals of $\tilde{f}(x)$ and $\tilde{Z}^{2}(x)$ for the two types of FBSDEs.

\subsection{Confidence interval for $\tilde{f}(x)$}
Denote
\begin{flalign*}
g_{f_i}(x,h,\theta)=K\left(\frac{X_{i}-x}{h}\right)
\left(\frac{Y_{(i+1)\Delta}-Y_{i\Delta}}{\Delta}-\theta\right),
\end{flalign*}
where $\theta$ is the candidate value of the target quantity $\tilde{f}(x)$.
Define
\begin{flalign}
\label{optimizationproblem}
L_{f}(x,h,\theta) =\max _{\left(p_{1},\cdots,p_{n}\right)}\left\{\prod_{i=1}^{n} np_{i} \mid \sum_{i=1}^{n} p_{i} g_{f_i}(x,h,\theta)=0, p_{i} \geq 0, \sum_{i=1}^{n} p_{i}=1\right\} ,
\end{flalign}
and
\begin{flalign}
l_{f}(x,h,\theta)=-2\log L_{f}(x,h,\theta).
\end{flalign}
The following theorem describes the asymptotic properties of $l_{f}(x,h,\theta)$ which helps to construct the confidence interval of $\tilde{f}(x)$.
\begin{theorem}
\label{thmempiricalf}
Let Assumption $\ref{assum-kernel function}$ - $\ref{assumf}$ hold. Furthermore, 
$$h\ell(T,x)\stackrel {a.s.} {\longrightarrow}\infty,$$
and,
$$h^{3}\ell(T,x)\stackrel {a.s.} {\longrightarrow}0,$$
as $n,T\to \infty$. Then,
\begin{flalign*}
l_{f}(x,h,\tilde{f} (x)) \Rightarrow \chi^{2}(1),
\end{flalign*}
and,
\begin{flalign*}
l_{f}\left(x,h,\tilde{f}(x)+\frac{\tau(x)}{\sqrt{h\ell(T,x)} }\right) \Rightarrow
\chi^{2}\left(1,\frac{\tau^{2}(x)}{l(K_{2})\tilde{Z}^{2}(x)} \right),
\end{flalign*}
as $n,T\to \infty$, where $\tau(x)$ is fixed and $\chi^{2}(p_1, p_2)$ is the chi-squared distribution with degree of freedom $p_1$ and non-central parameter $p_2$.
Then the $100(1-\alpha)\%$ EL confidence interval for $\tilde{f}(x)$ is defined as
\begin{flalign*}
I_{f,\alpha}^{EL}=\{l_{f}(x,h,\theta)\leq \chi_{1-\alpha}^{2}(1) \},
\end{flalign*}
where $\chi_{1-\alpha}^{2}(1)$ is the inverse cumulative distribution function for the $\chi^{2}(1)$
distribution evaluated at $1-\alpha$.
\end{theorem}

\begin{remark}
(i) Here we focus on local constant smoothing, although the extensions to the local linear (polynomial) are entirely straightforward.    \\
(ii) $l_{f}(x,h,\theta)$ can be calculated easily.
Using Lagrange multipliers, the constrained optimization problem $(\ref{optimizationproblem})$ is solved by
\begin{equation}
\label{lambda}
\widehat{p}_i=\frac{1}{n\left(1+\lambda_{f}(x, h, \theta) g_{f_i}(x, h, \theta)\right)},
\end{equation}
where $\lambda_{f}(x, h, \theta)$ satisfies
\begin{equation}
\label{solveprolambdasum}
\sum_{i=1}^n \frac{g_{f_i}(x, h, \theta)}{1+\lambda_{f}(x, h, \theta) g_{f_i}(x, h, \theta)}=0 .
\end{equation}
Then, given $\theta$, we can solve $(\ref{solveprolambdasum})$ numerically, and then compute
$l_{f}(x, h, \theta)=2\sum_{i=1}^{n-1}\log\left(1+\lambda_{f} (x, h, \theta) g_{f_i}(x, h, \theta)\right)$.
\end{remark}

\subsection{Confidence interval for $\tilde{Z}^{2}(x)$}
The point-wise EL confidence interval for $\tilde{Z}^{2}(x)$ can be constructed in a similar way.
Denote
\begin{flalign*}
g_{Z_i}(x,h,\theta)=K\left(\frac{X_{i}-x}{h}\right)
\left(\frac{(Y_{(i+1)\Delta}-Y_{i\Delta})^{2}}{\Delta}-\theta\right),
\end{flalign*}
and,
\begin{flalign}
L_{Z}(x,h,\theta) =\max _{\left(p_{1},\cdots,p_{n}\right)}\left\{\prod_{i=1}^{n} np_{i} \mid \sum_{i=1}^{n} p_{i} g_{Z_i}(x,h,\theta)=0, p_{i} \geq 0, \sum_{i=1}^{n} p_{i}=1\right\} ,
\end{flalign}
\begin{flalign}
l_{Z}(x,h,\theta)=-2\log L_{Z}(x,h,\theta).
\end{flalign}
The following theorem describes the asymptotic properties of $l_{Z}(x,h,\theta)$ which helps to construct the confidence interval of $\tilde{Z}^{2}(x)$.
\begin{theorem}
\label{thmempiricalZ}
Let Assumption $\ref{assum-kernel function}$ - $\ref{assumz}$ hold. Furthermore, 
$$h\ell(T,x)\stackrel {a.s.} {\longrightarrow}\infty,$$
and,
$$\frac{h^{3}\ell(T,x)}{\Delta}\stackrel {a.s.} {\longrightarrow}0,$$
as $n,T\to \infty$. Then,
\begin{flalign*}
l_{Z}(x,h,\tilde{Z}^{2}(x)) \Rightarrow \chi^{2}(1),
\end{flalign*}
and,
\begin{flalign*}
l_{Z}\left(x,h,\tilde{Z}^{2}(x)+\frac{\tau(x)}{\sqrt{h\ell(T,x)} }\right) \Rightarrow
\chi^{2}\left(1,\frac{\tau^{2}(x)}{l(K_{2})\tilde{Z}^{4}(x)} \right),
\end{flalign*}
as $n,T\to \infty$.
Then the $100(1-\alpha)\%$ EL confidence interval for $\tilde{Z}^{2}(x)$ is defined as
\begin{flalign*}
I_{Z,\alpha}^{EL}=\{l_{Z}(x,h,\theta)\leq \chi_{1-\alpha}^{2}(1) \}.
\end{flalign*}
\end{theorem}

\section{Proofs}
To give the proofs the main theorems, we introduce some useful lemmas at first.
\begin{lemma}(Lemma $6$ and Lemma $10$ in \cite{ait2016})
\label{lemma-discrete-continuous-K}
\begin{flalign}
\frac{\Delta}{h}\sum\limits_{i=0}^{n-1}K\left(\frac{X_{i\Delta}-x}{h}\right)
&= \frac{1}{h} \int_{0}^{T} K\left(\frac{X_{t}-x}{h}\right) dt+O_{p}\left(h^{-2}\Delta\ell(T,x)\right)\nonumber\\
&=\ell(T,x)+o_{p}(\ell(T,x))+O_{p}\left(h^{-2}\Delta\ell(T,x)\right)\nonumber
\end{flalign}
uniformly in $T$ as $h \rightarrow 0$ and $\Delta \rightarrow 0$.
\end{lemma}

\begin{lemma}
\label{lemma-K*g-discrete-continuous}
Assume that $w$ is continuously differentiable on $\mathcal{D}$. Then we have
\begin{flalign*}
\frac{\Delta}{h} \sum_{i=0}^{n-1} K\left(\frac{X_{i\Delta}-x}{h}\right) w\left(X_{i \Delta}\right)
&=\frac{1}{h}\int_{0}^{T}  K\left(\frac{X_{s}-x}{h}\right) w\left(X_{s}\right)ds
+O_{p}\left(h^{-2}\Delta T_{w}\ell(T,x)\right)%
+O_{p}\left(\kappa_{2}T_{w'}\ell(T,x)\right)\nonumber\\
&=w(x)\ell(T,x) +o_{p}(\ell(T,x))+O_{p}\left(h^{-2}\Delta T_{w}\ell(T,x)\right)%
+O_{p}\left(\kappa_{2}T_{w'}\ell(T,x)\right)
\end{flalign*}
uniformly in $T$ as $h \rightarrow 0$ and $\Delta \rightarrow 0$, where $\kappa_{2}=\sup\limits_{t,s\in[0,T],|t-s|\leq\Delta}|X_{t}-X_{s}|
$. $T_{\tilde{\theta}}=\max\limits_{0\leq t\leq T}|\tilde{\theta}(X_{t})|$ for a function $\tilde{\theta}$.
\end{lemma}
\begin{proof}
For the first equality, applying the triangular inequality and the mean value theorem, we have
\begin{flalign}
&\frac{\Delta}{h} \sum_{i=0}^{n-1} K\left(\frac{X_{i\Delta}-x}{h}\right) w\left(X_{i \Delta}\right)-\frac{1}{h}\int_{0}^{T}  K\left(\frac{X_{s}-x}{h}\right) w\left(X_{s}\right)ds  \nonumber\\
&\leq\frac{1}{h}\left|\sum\limits_{i=0}^{n-1} \int_{i\Delta}^{(i+1) \Delta}\left(K\left(\frac{X_{s}-x}{h}\right) -K\left(\frac{X_{i\Delta}-x}{h}\right)\right)w\left(X_{i\Delta}\right)
ds\right| \label{lemma3.7-proof-first part}\\
&~~+\frac{1}{h}\left|\sum\limits_{i=0}^{n-1} \int_{i\Delta}^{(i+1)\Delta} K\left(\frac{X_{s}-x}{h}\right)w'(\xi)(X_{s}-X_{i\Delta})ds\right|,\label{xi}
\end{flalign}
where $\xi$ in $(\ref{xi})$ is a value on the line segment connecting $X_s$ to $X_{i\Delta}$.
Using It\^o formula, we could get that $(\ref{lemma3.7-proof-first part})$ is of $O_{p}\left(h^{-2}\Delta T_{w}\ell(T,x)\right)$
uniformly in $T$ as $h \rightarrow 0$ and $\Delta \rightarrow 0$. Furthermore,
$(\ref{xi})$ is bounded by
\begin{flalign*}
\frac{\kappa_{2}T_{w'}}{h}\int_{0}^{T} \left|K\left(\frac{X_{s}-x}{h}\right)\right|ds
=\kappa_{2} T_{w'} \ell(T,x)\int_{-\infty}^{\infty}|K(u)|du
+o_{p}\left(\kappa_{2}T_{w'}\ell(T,x)\right)
\end{flalign*}
uniformly in $T$ as $h \rightarrow 0$ and $\Delta \rightarrow 0$.
\end{proof}

\begin{lemma}
\label{lemma-K*(g(X_s)-g(X_i))}
Let $w$ be continuously differentiable on $\mathcal{D}$. Then we have
\begin{flalign*}
&\frac{1}{h} \sum_{i=0}^{n-1} K\left(\frac{X_{i\Delta}-x}{h}\right)\int_{i\Delta}^{(i+1)\Delta}\left( w\left(X_{s}\right)-w\left(X_{i \Delta}\right)\right)ds
=O_{p}\left(\kappa_{2}T_{w'}\ell(T,x)\right)
\end{flalign*}
uniformly in $T$ as $h \rightarrow 0$ and $\Delta \rightarrow 0$.
\end{lemma}

\begin{lemma}
\label{lemmaKmintgdt}
Let  $m$, $w$ be continuously differentiable on $\mathcal{D}$. Then we have
\begin{flalign*}
\frac{1}{h}& \sum_{i=0}^{n-1} K\left(\frac{X_{i\Delta}-x}{h}\right)
m\left(X_{i\Delta}\right)\int_{i\Delta}^{(i+1)\Delta}w(X_{s})ds\\
&=(m\cdot w)(x)\ell(T,x)
+O_{p}\left(h^{-2}\Delta T_{m\cdot w}\ell(T,x)\right)
+O_{p}\left(\kappa_{2}T_{(m\cdot w)'}\ell(T,x)\right)
+O_{p}(\kappa_{2}T_{w'}\ell(T,x))\\
&~~+O_{p}\left(h^{-2}\kappa_{2}\Delta T_{m}T_{w'}\ell(T,x)\right)
+
O_{p}\left(\kappa^{2}_{2}T_{m'}T_{w'}\ell(T,x)\right)
\end{flalign*}
uniformly in $T$ as $h \rightarrow 0$ and $\Delta \rightarrow 0$.
\end{lemma}
\begin{proof}
We write
\begin{flalign*}
&\frac{1}{h} \sum_{i=0}^{n-1} K\left(\frac{X_{i\Delta}-x}{h}\right)
m\left(X_{i\Delta}\right)\int_{i\Delta}^{(i+1)\Delta}w(X_{s})ds\\
&=\frac{\Delta}{h} \sum_{i=0}^{n-1} K\left(\frac{X_{i\Delta}-x}{h}\right)
m\left(X_{i\Delta}\right)w(X_{i\Delta})+\frac{1}{h} \sum_{i=0}^{n-1} K\left(\frac{X_{i\Delta}-x}{h}\right)m\left(X_{i\Delta}\right)\int_{i\Delta}^{(i+1)\Delta}
\left(w(X_{s})-w(X_{i\Delta})\right)ds,
\end{flalign*}
then the result could be obtained with the application of Lemma $\ref{lemma-K*g-discrete-continuous}$ and Lemma $\ref{lemma-K*(g(X_s)-g(X_i))}$.
\end{proof}

\begin{lemma}
\label{lemmakww}
Assume that $w$ is continuously differentiable on $\mathcal{D}$. Then we have,
\begin{flalign*}
\frac{1}{h}\int_{0}^{T}K\left(\frac{X_{s}-x}{h}\right)
\left(w(X_{s})-w(x)\right)ds
=l(K_{1})w'(x)\frac{s'(x)}{s(x)}h^{2}\ell(T,x)
+ O_p(h^{1/2}\ell(T,x)^{1/2})+o_{p}(h\ell(T,x)).
\end{flalign*}
\end{lemma}

\begin{lemma}
\label{lemmaukww}
Let $w$ is continuously differentiable on $\mathcal{D}$. Then we have,
\begin{flalign*}
\frac{1}{h}\int_{0}^{T}\frac{X_{s}-x}{h}K\left(\frac{X_{s}-x}{h}\right)
\left(w(X_{s})-w(x)\right)ds
=l(K_{1}) w'(x)h\ell(T,x) + o_p(h\ell(T, x)).
\end{flalign*}
\end{lemma}

\begin{proof}[Proof of Theorem $\ref{thm-g-NW}$]
The proof is analogous to the proof of Theorem \ref{thm-z-NW}, we omit it here.
\end{proof}

\begin{proof}[Proof of Theorem $\ref{thm-g-LL}$]
Similar to the proof of $\ref{lemma-discrete-continuous-K}$, we can obtain that
\begin{flalign}
\label{proofDK2}
D(K_{2})=l(K_{1})\ell(T,x)+o_{p}(\ell(T,x))+O_{p}\left(\frac{\Delta}{h^{2}}\ell(T,x)\right),
\end{flalign}
uniformly in $T$ as $h\to 0$ and $\Delta\to 0$. It follows from Lemma $\ref{lemmakww}$ with $w=1$ that
\begin{flalign}
\label{proofDK1}
D(K_{1})=l(K_{1})\frac{s'(x)}{s(x)}h\ell(T,x) + o_p(h\ell(T, x))
\end{flalign}
uniformly in $T$ as $h\to 0$ and $\Delta\to 0$. Then, we can get,
\begin{flalign*}
D(K)D(K_{2})-D(K_{1})^{2}=l(K_{1})\ell(T,x)^{2}+o_{p}(\ell(T,x)^{2})
\end{flalign*}
uniformly in $T$ as $h\to 0$ and $\Delta\to 0$.
Applying Lemma $\ref{lemma-K*g-discrete-continuous}$ and $\ref{lemmaukww}$ with $w=\tilde{f}$, we can deduce that
\begin{flalign}
B(K_{1},\tilde{f})
&=\frac{1}{h}\int_{0}^{T}\frac{X_{s}-x}{h}K\left(\frac{X_{s}-x}{h}\right)
\left(\tilde{f}(X_{s})-\tilde{f}(x)\right)ds
+O_{p}(h^{-2}\Delta T_{\tilde{f}}\ell(T,x))+O_{p}(\kappa_{2}T_{\tilde{f}'}\ell(T,x)) \nonumber\\
&=l(K_{1})\tilde{f}'(x)h\ell(T,x)+o_{p}(h\ell(T,x))
+O_{p}(h^{-2}\Delta T_{\tilde{f}}\ell(T,x))+O_{p}(\kappa_{2}T_{\tilde{f}'}\ell(T,x)) \label{proofBK1F}
\end{flalign}
uniformly in $T$ as $h\to 0$ and $\Delta\to 0$. According to Lemma $B.2$ in \cite{Kim2017}, $\kappa_{2}
=O_{p}(\Delta T_{b})+O_{p}\left(T_{\sigma}(\Delta\log n)^{1/2}\right) $. Moreover, combining $(\ref{proofDK2})$, $(\ref{proofDK1})$ and $(\ref{proofBK1F})$, we have,
\begin{flalign*}
&B(K,\tilde{f})D(K_{2})-B(K_{1},\tilde{f})D(K_{1})  \\
&=
\left[l(K_{1})\tilde{f}'(x)\frac{s'(x)}{s(x)}h^{2}\ell(T,x)
+ O_p(h^{1/2}\ell(T,x)^{1/2})+o_{p}(h\ell(T,x))+O_{p}(\kappa_{2}T_{\tilde{f}'}\ell(T,x))\right] \\
&~~\times\left[l(K_{1})\ell(T,x)+o_{p}(\ell(T,x))+O_{p}\left(\frac{\Delta}{h^{2}}\ell(T,x)\right)\right] \\
&~~-\left[l(K_{1})\tilde{f}'(x)h\ell(T,x)+o_{p}(h\ell(T,x))+O_{p}(\kappa_{2}T_{\tilde{f}'}\ell(T,x))\right]
\left[l(K_{1})\frac{s'}{s}(x)h\ell(T,x) + o_p(h\ell(T, x)) \right]\\
&=O_p(h^{1/2}\ell(T,x)^{3/2})+o_p(h\ell(T,x)^{2})
\end{flalign*}
uniformly in $T$ as $h\to 0$ and $\Delta\to 0$.  In consequence, we have,
\begin{flalign*}
\hat{g}_{p,\text{LL}}(x)-\tilde{f}(x)
=o_p(h)+O_p(h^{1/2}\ell(T,x)^{-1/2})
\end{flalign*}
uniformly in $T$ as $h\to 0$ and $\Delta\to 0$.

For $\hat{g}_{q,\text{LL}}(x)$, following the following proof of $M(K,Z)$, we have,
\begin{flalign*}
M(K_{1},\tilde{f})=O_p(h^{-1/2}\ell(T,x)^{1/2})
\end{flalign*}
uniformly in $T$ as $h\to 0$ and $\Delta\to 0$. Then,
\begin{flalign*}
M(K_{1},\tilde{f})D(K_{1})
=O_p(h^{-1/2}\ell(T,x)^{1/2})O_p(h\ell(T,x))
=O_p(h^{1/2}\ell(T,x)^{3/2})
\end{flalign*}
uniformly in $T$ as $h\to 0$ and $\Delta\to 0$. Therefore, we have,
\begin{flalign*}
&(h\ell(T,x))^{1/2}\hat{g}_{q,\text{LL}}(x) \\
&=(h\ell(T,x))^{1/2}\frac{M(K,\tilde{f})D(K_{2})-M(K_{1},\tilde{f})D(K_{1})}{D(K)D(K_{2})-D(K_{1})^2}   \\
&=(h\ell(T,x))^{1/2}\frac{l(K_{1})\ell(T,x)(1+o_{p}(1))}{l(K_{1})\ell(T,x)^{2}(1+o_{p}(1))}M(K,\tilde{f})
+\frac{O_p(h^{1/2}\ell(T,x)^{3/2})}{O_p(\ell(T,x)^{2})} \\
&=\frac{O_p(h^{1/2}M(K,\tilde{f}))}{\ell(T,x)^{1/2}}+o_p(1).
\end{flalign*}
uniformly in $T$ as $h\to 0$ and $\Delta\to 0$. As shown in the proof of $M(K,Z)$, we can get,
\begin{flalign*}
\frac{h^{1/2}M(K,\tilde{f})}{\ell(T,x)^{1/2}}\Rightarrow l(K_{2})^{1/2}\tilde{Z}(x)\mathbb{N}.
\end{flalign*}

Similar to the proof of Lemma \ref{lemma-K*(g(X_s)-g(X_i))}, we can get
\begin{flalign*}
R(K_{1},\tilde{f})=o_{p}(h^{2}\ell(T,x))
\end{flalign*}
uniformly in $T$ as $h\to 0$ and $\Delta\to 0$. Then,
\begin{flalign*}
R(K,\tilde{f})D(K_{2})-R(K_{1},\tilde{f})D(K_{1})
=o_{p}(h^{2}\ell(T,x)^{2})+o_{p}(h^{3}\ell(T,x)^{2})
\end{flalign*}
uniformly in $T$ as $h\to 0$ and $\Delta\to 0$. It follows that
\begin{flalign*}
\hat{g}_{r,\text{LL}}(x)=o_{p}(h^{2}). 
\end{flalign*}
Then the proof is complete.
\end{proof}

\begin{proof}[Proof of Theorem $\ref{thm-z-NW}$]
Applying Lemma $\ref{lemma-K*g-discrete-continuous}$ and $\ref{lemmakww}$, we have,
\begin{flalign*}
&\frac{\Delta}{h} \sum_{i=0}^{n-1} K\left(\frac{X_{i\Delta}-x}{h}\right) \left(\tilde{Z}^{2}(X_{i\Delta})-\tilde{Z}^{2}(x)\right)        \\
&=\frac{1}{h}\int_{0}^{T}  K\left(\frac{X_{s}-x}{h}\right) \left(\tilde{Z}^{2}(X_{s})-\tilde{Z}^{2}(x)\right)ds
+O_{p}\left(h^{-2}\Delta T_{\tilde{Z}^{2}}\ell(T,x)\right)
+O_{p}\left(\kappa_{2}T_{(\tilde{Z}^{2})'}\ell(T,x)\right)      \\
&=(\tilde{Z}^{2})'(x)hl(K_{1})\frac{s'(x)}{s(x)}h\ell(T,x)
+o_{p}(h\ell(T,x))+O_{p}(h^{1/2}\ell(T,x)^{1/2})    \\
&~~+O_{p}\left(h^{-2}\Delta T_{\tilde{Z}^{2}}\ell(T,x)\right)
+O_{p}\left(\kappa_{2}T_{(\tilde{Z}^{2})'}\ell(T,x)\right),
\end{flalign*}
uniformly in $T$ as $h\to 0$ and $\Delta\to 0$.
Then, $(\ref{proofthmzp})$ can be obtained. For the second part $\hat{Z}^{2}_{q,\text{NW}}(x)$, we analyse the numerator (denoted as $A(x)$) first,  define a continuous martingale $M$ such that
\begin{flalign}
&M_{T}=\sqrt{\frac{4}{h\Delta}} \sum_{i=0}^{n-1} K\left(\frac{X_{i\Delta}-x}{h}\right) \int_{i\Delta}^{(i+1)\Delta}(Y_{t}-Y_{i\Delta})Z_{t}dW_{t}.
\end{flalign}
Note that $M_{T}=\sqrt{\frac{h}{\Delta}}A(x)$. Using It\^o formula, we can deduce that
\begin{flalign*}
[M]_{T}&=\frac{4}{h\Delta} \sum_{i=0}^{n-1} K^{2}\left(\frac{X_{i\Delta}-x}{h}\right) \int_{i\Delta}^{(i+1)\Delta} (Y_{t}-Y_{i\Delta})^{2}\tilde{Z}^2(X_{t})dt   \\
&=\frac{4}{h\Delta} \sum_{i=0}^{n-1} K^{2}\left(\frac{X_{i\Delta}-x}{h}\right) \bigg[
\tilde{Z}^2\left(X_{i\Delta}\right) \int^{(i+1)\Delta}_{i\Delta}\left(Y_{t}-Y_{i \Delta}\right)^{2}dt
+\int^{(i+1)\Delta}_{i \Delta}\left(Y_{t}-Y_{i \Delta}\right)^{2}\left(\tilde{Z}^2\left(X_{t}\right)-\tilde{Z}^2(X_{i\Delta})\right)dt\bigg]   \\
&=\frac{2\Delta}{h} \sum_{i=0}^{n-1} K^{2}\left(\frac{X_{i\Delta}-x}{h}\right)\tilde{Z}^{4}\left(X_{i\Delta}\right)  \\
&~~+\frac{8}{h\Delta} \sum_{i=0}^{n-1} K^{2}\left(\frac{X_{i\Delta}-x}{h}\right)\tilde{Z}^{2}\left(X_{i\Delta}\right)\int^{(i+1)\Delta}_{i \Delta}\left((i+1)\Delta-t\right)\left(Y_{t}-Y_{i \Delta}\right) \tilde{f}\left(X_{t}\right)dt    \\
&~~+\frac{4}{h\Delta} \sum_{i=0}^{n-1} K^{2}\left(\frac{X_{i\Delta}-x}{h}\right)\tilde{Z}^2\left(X_{i\Delta}\right)\int^{(i+1)\Delta}_{i \Delta}\left((i+1)\Delta-t\right)\left(\tilde{Z}^2(X_{t})-\tilde{Z}^2(X_{i\Delta})\right)dt   \\
&~~+\frac{8}{h\Delta} \sum_{i=0}^{n-1} K^{2}\left(\frac{X_{i\Delta}-x}{h}\right)\tilde{Z}^2\left(X_{i\Delta}\right) \int^{(i+1)\Delta}_{i \Delta}\left((i+1)\Delta-t\right)\left(Y_{t}-Y_{i \Delta}\right) \tilde{Z}(X_{t})dW_{t}     \\
&~~+\frac{4}{h\Delta} \sum_{i=0}^{n-1} K^{2}\left(\frac{X_{i\Delta}-x}{h}\right)
\int^{(i+1)\Delta}_{i \Delta}\left(Y_{t}-Y_{i \Delta}\right)^{2}\left(\tilde{Z}^2\left(X_{t}\right)-\tilde{Z}^2(X_{i\Delta})\right)dt   \\
&=:A_{1}+A_{2}+A_{3}+A_{4}+A_{5},
\end{flalign*}
which will show in detail as follows.

With the application of Lemma $\ref{lemma-K*g-discrete-continuous}$, we have,
\begin{flalign*}
A_{1}
&=\frac{2}{h}\int_{0}^{T}K^{2}\left(\frac{X_{t}-x}{h}\right)\tilde{Z}^{4}\left(X_{t}\right)dt
+O_{p}(h^{-2}\Delta T_{\tilde{Z}^{4}}\ell(T,x))
++O_{p}(\kappa_{2}T_{(\tilde{Z}^{4})'}\ell(T,x))   \\
&=\tilde{Z}^{4}(x)l(K_{2})\ell(T,x)(1+o_{p}(1))
\end{flalign*}
uniformly in $T$ as $h\to 0$ and $\Delta\to 0$.
Applying Lemma \ref{lemmaKmintgdt}, we can get,
\begin{flalign*}
A_{2}&=O_{p}\left(\kappa_{1}\ell(T,x)\right)
+O_{p}\left(\frac{\Delta}{h^{2}}\kappa_{1}T_{\tilde{Z}^2|\tilde{f}|}\ell(T,x)\right)
+O_{p}\left(\kappa_{1}\kappa_{2}T_{(\tilde{Z}^2|\tilde{f}|)'}\ell(T,x)\right)
+O_{p}\left(\kappa_{1}\kappa_{2}T_{|\tilde{f}|'}\ell(T,x)\right)      
\end{flalign*}
uniformly in $T$ as $h\to 0$ and $\Delta\to 0$, where $\kappa_{1}=\sup\limits_{t,s\in[0,T],|t-s|\leq\Delta}|Y_{t}-Y_{s}|
=O_{p}(\Delta T_{\tilde{f}})+O_{p}\left(T_{\tilde{Z}}(\Delta\log n)^{1/2}\right) $.
$A_{3}$ can be bounded by
\begin{flalign*}
&4\kappa_{2}T_{(\tilde{Z}^2)'}\frac{\Delta}{h}\sum_{i=0}^{n-1} K^{2}\left(\frac{X_{i\Delta}-x}{h}\right)\tilde{Z}^2\left(X_{i\Delta}\right)  \\
&=O_{p}\left(\kappa_{2}T_{(\tilde{Z}^2)'}\ell(T,x)\right)
+O_{p}\left( \frac{\Delta}{h^{2}}\kappa_{2}T_{(\tilde{Z}^2)'} T_{\tilde{Z}^2}\ell(T,x) \right)
+O_{p}\left(\kappa_{2}^{2}T_{(\tilde{Z}^2)'}^{2}\ell(T,x) \right).
\end{flalign*}
uniformly in $T$ for all $x\in \mathcal{D}$ as $\Delta \rightarrow 0$ and $h \rightarrow 0$ with the application of Lemma $\ref{lemma-K*g-discrete-continuous}$.
Similarly, we have,
\begin{flalign*}
A_{4}
&=O_{p}\left(\frac{\kappa_{1}}{h^{1/2}}\ell(T,x)^{1/2}\right)
+O_{p}\left(\frac{\Delta^{1/2}}{h^{3/2}} \kappa_{1} T_{\tilde{Z}^{3}}\ell(T,x)^{1/2} \right)
+O_{p}\left(\frac{\kappa_{1}\kappa_{2}^{1/2}}{h^{1/2}} T_{(\tilde{Z}^{6})'}^{1/2}\ell(T,x)^{1/2}\right)    \\
&~~+O_{p}\left(\frac{\kappa_{1}\kappa_{2}^{1/2}}{h^{1/2}} T_{(\tilde{Z}^2)'}^{1/2}\ell(T,x)^{1/2}\right)
+O_{p}\left( \frac{\Delta^{1/2}\kappa_{1}\kappa_{2}^{1/2}}{h^{3/2}} T_{(\tilde{Z}^2)'}^{1/2}T_{\tilde{Z}^2}\ell(T,x)^{1/2}\right)
+O_{p}\left(\frac{\kappa_{1}\kappa_{2}}{h^{1/2}} T_{(\tilde{Z}^2)'}^{1/2}T_{(\tilde{Z}^{4})'}^{1/2}\ell(T,x)^{1/2}\right)  \\
&=o_{p}\left(\ell(T,x)\right)
\end{flalign*}
uniformly in $T$ as $\Delta \rightarrow 0$ and $h \rightarrow 0$.
For $A_{5}$, it can be bounded by
\begin{flalign*}
&\frac{4\kappa_{1}T_{(\tilde{Z}^2)'}}{h\Delta} \sum_{i=0}^{n-1} K^{2}\left(\frac{X_{i\Delta}-x}{h}\right)
\int^{(i+1)\Delta}_{i \Delta}\left(Y_{t}-Y_{i \Delta}\right)^{2}dt \\
&=o_{p}\left(\frac{1}{h\Delta} \sum_{i=0}^{n-1} K^{2}\left(\frac{X_{i\Delta}-x}{h}\right)
\int^{(i+1)\Delta}_{i \Delta}\left(Y_{t}-Y_{i \Delta}\right)^{2}dt\right)
=o_{p}\left(\ell(T,x)\right)
\end{flalign*}
uniformly in $T$ as $\Delta \rightarrow 0$ and $h \rightarrow 0$.
Therefore, we have,
$[M]_{T}=\tilde{Z}^{4}(x)l(K_{2})\ell(T,x)(1+o_{p}(1))$ uniformly in $T$ for all $x\in \mathcal{D}$ as $\Delta \rightarrow 0$ and $h \rightarrow 0$.

Moreover, we have
\begin{flalign*}
&[W,M]_{T}  \\
&=\sqrt{\frac{4}{h \Delta}} \sum_{i=0}^{n-1} K\left(\frac{X_{i \Delta}-x}{h}\right) \int^{(i+1)\Delta}_{i\Delta}\left(Y_{t}-Y_{i\Delta}\right) \tilde{Z}\left(X_{t}\right) dt  \\
&=\sqrt{\frac{4}{h \Delta}} \sum_{i=0}^{n-1} K\left(\frac{X_{i\Delta}-x}{h}\right) \tilde{Z}(X_{i\Delta})\int_{i\Delta}^{(i+1)\Delta}
((i+1)\Delta-t)\tilde{f}\left(X_{t}\right)dt        \\
&~~+\sqrt{\frac{4}{h \Delta}} \sum_{i=0}^{n-1} K\left(\frac{X_{i\Delta}-x}{h}\right) \tilde{Z}(X_{i\Delta})\int_{i\Delta}^{(i+1)\Delta}
((i+1)\Delta-t)\tilde{Z}\left(X_{t}\right)dW_{t}
\\
&~~+\sqrt{\frac{4}{h \Delta}} \sum_{i=0}^{n-1}K\left(\frac{X_{i\Delta}-x}{h}\right) \int_{i\Delta}^{(i+1)\Delta} \left(Y_{t}-Y_{i\Delta}\right) \left(\tilde{Z}(X_{t})-\tilde{Z}(X_{i\Delta})\right)dt \\
&=O_{p}\left(h^{1/2}\Delta^{1/2}  \ell(T,x)\right)
+O_{p}\left(h^{-3/2}\Delta^{3/2}T_{\tilde{f}} T_{\tilde{Z}}\ell(T,x)\right)
+O_{p}\left(h^{1/2}\Delta^{1/2}\kappa_{2}T_{(\tilde{Z}\cdot |\tilde{f}|)'}\ell(T,x)\right)    \\
&~~+O_{p}\left(h^{1/2}\Delta^{1/2}\kappa_{2}T_{|\tilde{f}|'}  \ell(T,x)\right)
+O_{p}\left(h^{-3/2}\Delta^{3/2}\kappa_{2}T_{|\tilde{f}|'} T_{\tilde{Z}}\ell(T,x)\right)
+O_{p}\left(h^{1/2}\Delta^{1/2}\kappa^{2}_{2}T_{\tilde{Z}'}T_{|\tilde{f}|'}\ell(T,x)\right)    \\
&~~+O_{p}\left(\Delta^{1/2} \ell(T,x)^{1/2}\right)
+O_{p}\left(\frac{\Delta}{h}T_{\tilde{Z}^2}\ell(T,x)^{1/2}\right)
+O_{p}\left(\Delta^{1/2}\kappa_{2}^{1/2}T^{1/2}_{(\tilde{Z}^2)'} \ell(T,x)^{1/2}\right)  \\
&~~+O_{p}\left(\frac{\Delta}{h}\kappa_{2}^{1/2}T_{\tilde{Z}}T^{1/2}_{(\tilde{Z}^2)'} \ell(T,x)^{1/2}\right)
+O_{p}\left(\Delta^{1/2}\kappa_{2}T_{(\tilde{Z}^2)'}\ell(T,x)^{1/2}\right)    \\
&~~+O_{p}\left(h^{1/2}\Delta^{1/2}\kappa_{1}\kappa_{2}T_{\tilde{Z}'}  \ell(T,x)\right)        \\
&=o_{p}\left(h^{1/2}\ell(T,x)\right)
\end{flalign*}
uniformly in $T$ for all $x\in \mathcal{D}$ as $\Delta \rightarrow 0$ and $h \rightarrow 0$.
Therefore,
\begin{flalign*}
&\frac{[W, M]_{T}}{[M]_{T}}=o_{p}(h^{1/2}),\ \
\text{and}\ \
\frac{[W, M]_{T}}{[W]_{T}}=o_{p}\left(h^{1/2}T^{-1}\right).
\end{flalign*}
Consequently, we can deduce that
\begin{flalign*}
&M_{T}\ell(T,x)^{-1/2}\Rightarrow\sqrt{2}l(K_{2})^{\frac{1}{2}}\tilde{Z}^2(x)\mathbb{N}.
\end{flalign*}
Then,
\begin{flalign*}
&\sqrt{\frac{h\ell(T,x)}{\Delta}} \hat{Z}_{q,\text{NW}}^{2}(x)=\frac{\sqrt{\ell(T,x)}}{\ell(T,x)(1+ o_{p}(1))}M_{T}\Rightarrow\sqrt{2}l(K_{2})^{\frac{1}{2}}\tilde{Z}^2(x)\mathbb{N}.
\end{flalign*}

For $\hat{Z}^{2}_{r,\text{NW}}(x)$, we can apply Lemma $\ref{lemma-K*(g(X_s)-g(X_i))}$ to the first part to obtain that
\begin{flalign*}
&\frac{ \frac{1}{h} \sum_{i=0}^{n-1} K\left(\frac{X_{i\Delta}-x}{h}\right) \int_{i\Delta}^{(i+1)\Delta}\left(\tilde{Z}^2(X_{t})-\tilde{Z}^2(X_{i\Delta})\right)dt  }{\frac{\Delta}{h}\sum_{i=0}^{n-1} K\left(\frac{X_{i\Delta}-x}{h}\right)}
=O_{p}(\kappa_{2}T_{(\tilde{Z}^2)'})=o_{p}(h^{2})
\end{flalign*}
uniformly in $T$ for all $x\in \mathcal{D}$ as $\Delta \rightarrow 0$ and $h \rightarrow 0$.
For the second part of $\hat{Z}^{2}_{r,\text{NW}}(x)$, it is easily to obtain that
\begin{flalign*}
&\frac{\frac{1}{h} \sum_{i=0}^{n-1} K\left(\frac{X_{i\Delta}-x}{h}\right) \int_{i\Delta}^{(i+1)\Delta}2(Y_{t}-Y_{i\Delta})\tilde{f}(X_{t})dt}{\frac{\Delta}{h}\sum_{i=0}^{n-1} K\left(\frac{X_{i\Delta}-x}{h}\right)}  \\
&=O_{p}(\kappa_{1})
+O_{p}\left(\frac{\Delta}{h^{2}}\kappa_{1}T_{\tilde{f}}\right)
+O_{p}(\kappa_{1}\kappa_{2}T_{|\tilde{f}|'})    \\
&=o_{p}(h^{2})
\end{flalign*}
uniformly in $T$ for all $x\in \mathcal{D}$ as $\Delta \rightarrow 0$ and $h \rightarrow 0$ with the application of Lemma $\ref{lemmaKmintgdt}$. Then the proof is complete.
\end{proof}

\begin{proof}[Proof of Theorem \ref{thmempiricalf}]
We only show the sketch of proof in the following.
Denote $\lambda_{f} (x, h, \theta)$ in $(\ref{lambda})$ as $\lambda_{f} $ for simplicity. It is easy to get,
\begin{flalign}
&l_f\left(x,h,\theta\right) 
=2\sum_{i=1}^{n-1}\log\left(1+\lambda_f g_{f_i}\left(x,h,\theta\right)\right) \nonumber\\
&=\left(\sum_{i=1}^{n-1} g_{f_i}^2(x,h,\theta)\right)^{-1}\left(\sum_{i=1}^{n-1} g_{f_i}(x,h,\theta)\right)^{2}+o_{p}(1),  \label{ldoubleg} 
\end{flalign}
for $\theta=\tilde{f}(x)+\frac{\tau(x)}{\sqrt{h\ell(T,x)}}$ with any fixed $\tau(x)$. Then,
\begin{flalign*}
l_f\left(x,h,\tilde{f}(x)\right) =\frac{B^{2}(x,h,\tilde{f}(x))}{A(x,h,\tilde{f}(x))} +o_{p}(1),    
\end{flalign*}
where
\begin{align}
&A(x, h, \tilde{f}(x))=\frac{\sum_{i=1}^{n-1} g_{f_i}^2(x,h,\tilde{f}(x))}
{\left(\sum_{i=1}^{n-1} K\left(\frac{X_{i \Delta}^2-x}{h}\right)\right)^{2}}, \quad \text{and,}\quad
B(x, h, \theta)=\frac{\sum_{i=1}^{n-1} g_{f_i}(x,h,\tilde{f}(x))}{\sum_{i=1}^{n-1} K\left(\frac{X_{i \Delta}^2-x}{h}\right)}.
\end{align}
According to Theorem $\ref{thm-g-NW}$, we have,
\begin{align}
&h\ell(T,x) A(x,h,\tilde{f}(x)) \stackrel{p}{\longrightarrow} l(K_{2})\tilde{Z}^{2}(x), \label{propertyA} \\
&\sqrt{h \ell(T,x)} B(x,h,\tilde{f}(x)) \Rightarrow l(K_{2})^{1/2}\tilde{Z}(x)\mathbb{N}.
\end{align}
Then the conclusion can be obtained.
\end{proof}

\section{Simulation}
In this section, we investigate the finite-sample performance of the estimators for two models by the following measures:
\begin{align*}
&\text{MAE}_{v}=\frac{1}{m}\sum_{i=1}^{m}\left|\frac{1}{L}\sum_{l=1}^{L}\hat{v}^{l}(x_{i})-v(x_{i})\right|,    \\ &\text{MSE}_{v}=\frac{1}{m}\sum_{i=1}^{m}\left|\frac{1}{L}\sum_{l=1}^{L}\hat{v}^{l}(x_{i})-v(x_{i})\right|^{2}.
\end{align*}
where $v$ is the estimated function. $l$ represents the $l$-th Monte Carlo replication, $L=1000$. $\{x_{i}\}_{i=1}^{m}$ are chosen uniformly to cover the
range of sample path of $X_t$.
We use the Epanechnikov kernel $K(u)=\frac{3}{4}(1-u^{2})I_{\{|u|\leq1\}}$ and the bandwidth $h$ is chosen applying the cross-validation rule.

\begin{example}[$X_{t}$ is recurrent]
Consider an infinite horizon FBSDE,
\begin{align}
\label{exa-1}
\left\{\begin{array}{l}
X_{t}=x_{0}+\int_{0}^{t}\sigma dW_{s}, \\
Y_{t}=Y_{T}+\int_{t}^{T}\left((2+\sigma^2)\sin X_{s}-2\sigma^{2}\cos X_{s}-Y_{s}+\sigma Z_{s}
\right)ds-\int_{t}^{T} Z_{s} dW_{s},\ \forall\ t,T>0.
\end{array}\right.
\end{align}
\end{example}
The solution to $(\ref{exa-1})$ is $Y_{t}=2\sin X_{t}$, $Z_{t}=2\sigma \cos X_{t}$. The true estimated functions are $-f(X_{t},Y_{t},Z_{t})=-(2+\sigma^2)\sin X_{t}+2\sigma^{2}\cos X_{t}+Y_{t}-\sigma Z_{t}$ and $Z_{t}^{2}=4\sigma^{2}  \cos^{2}  X_{t}$.
The parameters we take are $x_{0}=0.05$, $\sigma=0.5$. 

Tables $1$-$2$ report the MAE and MSE of the local constant and linear estimators for $-f$ and $Z^{2}$ given the estimated points $X_{t}=x_i$  with different observation time $T$ and $n$.  Figures $1(a)$ and $1(b)$ report the 95\% confidence
intervals for these estimators, based on the common method using the asymptotic normality of the estimators and the empirical likelihood method.
\begin{table}[H]
\centering
\caption{MAE and MSE of the estimators for $f$}
\begin{tabular}{ccccc}
\hline
                                     &    & $T=10,n=1000$ & $T=8,n=3000$ & $T=10,n=5000$ \\ \hline
\multirow{2}{*}{MAE$(\times 10^{1})$}& NW & 0.129         & 0.134        & 0.123         \\ 
                                     & LL & 0.192         & 0.293        & 0.192         \\ \hline
\multirow{2}{*}{MSE$(\times 10^{4})$}& NW & 2.466         & 2.450        & 2.171         \\ 
                                     & LL & 5.420         & 12.498       & 5.476         \\ \hline
\end{tabular}
\end{table}
\begin{table}[H]
\centering
\caption{MAE and MSE of the estimators for $Z^{2}$}
\begin{tabular}{ccccc}
\hline
                                     &    & $T=10,n=1000$ & $T=8,n=3000$ & $T=10,n=5000$ \\ \hline
\multirow{2}{*}{MAE$(\times 10^{1})$}& NW & 0.200         & 0.085        & 0.094         \\
                                     & LL & 0.037         & 0.081        & 0.031         \\ \hline
\multirow{2}{*}{MSE$(\times 10^{4})$}& NW & 4.026         & 0.723        & 0.885         \\  
                                     & LL & 0.140         & 0.653        & 0.100         \\ \hline
\end{tabular}
\end{table}
\begin{figure}[H]
\centering
\begin{minipage}[H]{0.4\textwidth}
\includegraphics[width=1.05\textwidth]{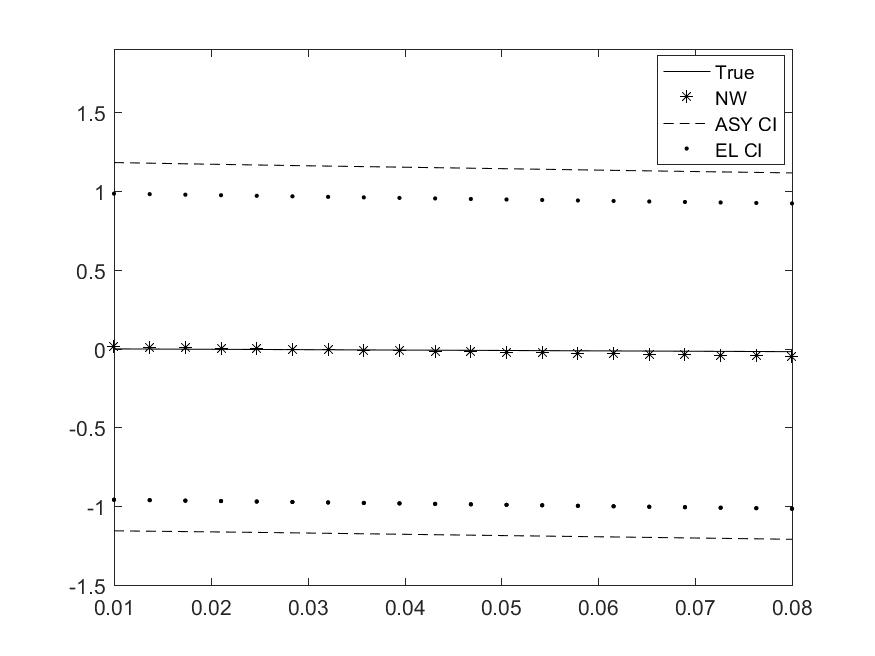}
\caption*{(a) Confidence interval of $f$}
\end{minipage}
\begin{minipage}[H]{0.4\textwidth}
\includegraphics[width=\textwidth]{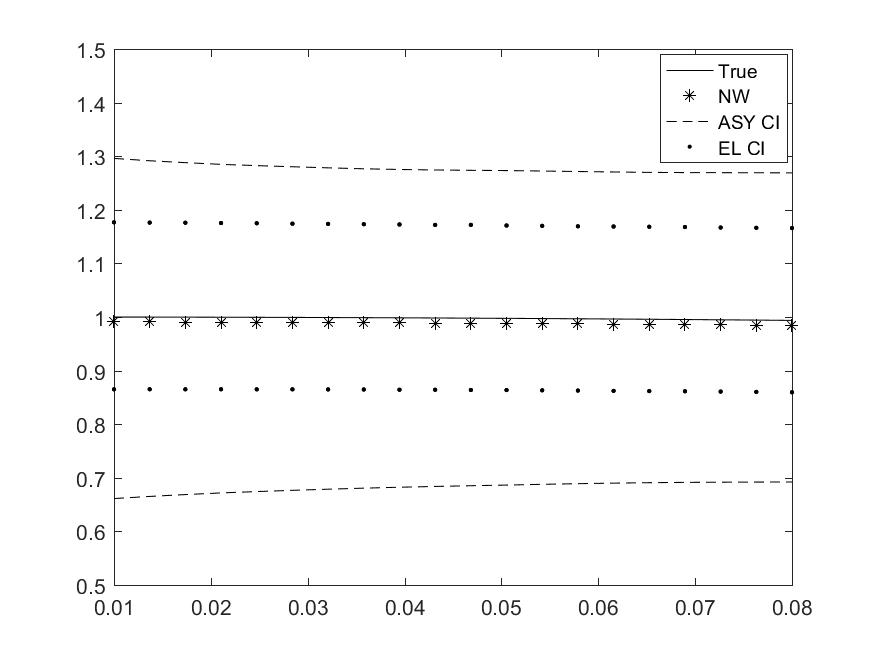}
\caption*{(b) Confidence interval of  $Z^{2}$}
\end{minipage}
\caption{The solid lines represent the true values, the '*' lines represent  the local constant estimators, the '--' dotted lines represent  the confidence intervals based on the common method using the asymptotic normality of the estimator,  the '$\cdot$' dotted lines represent the confidence intervals applying the empirical likelihood method. }
\end{figure}

\begin{example}
Consider an infinite horizon FBSDE,
\begin{align}
\label{exa-2}
\left\{
\begin{array}{l}
X_{t}=x_{0}+\int_{0}^{t} 0.02 X_{s}ds+\int_{0}^{t}0.2 X_{s} dW_{s}, \\
Y_{t}=Y_{T}+\int_{t}^{T}\left(0.04 X_{s}^{2}-0.4Y_{s}-1.1Z_{s} \right)ds-\int_{t}^{T} Z_{s} dW_{s},\ \forall\ t,T>0,
\end{array}
\right.
\end{align}
\end{example}
The solution to $(\ref{exa-2})$ is $Y_{t}=X_{t}^{2}$, $Z_{t}=-0.4 X_{t}^{2}$. We aim to estimate $-f(X_{t},Y_{t},Z_{t})=-0.04 X_{t}^{2}+0.4Y_{t}+1.1Z_{t} $ and $Z_{t}^{2}=0.16  X_{t}^{4}$.
$x_{0}=0.5$.

Tables $3$-$4$ report the MAE and MSE of the local constant and linear estimators for $f$ and $Z^{2}$ given the estimated points $X_{t}=x_i$ with different observation time interval $T$ and  $n$. Figures $2$ reports the 95\% confidence intervals based on the two methods.
\begin{table}[H]
\centering
\caption{MAE and MSE of the estimators for $f$}
\begin{tabular}{ccccc}
\hline
                                      &    & $T=10,n=1000$ & $T=8,n=3000$ & $T=10,n=5000$ \\ \hline
\multirow{2}{*}{MAE$(\times 10^{1})$} & NW & 0.533         & 0.482        & 0.533         \\
                                      & LL & 0.386         & 0.306        & 0.279         \\ \hline 
\multirow{2}{*}{MSE$(\times 10^{1})$} & NW & 0.028         & 0.023        & 0.028         \\  
                                      & LL & 0.016         & 0.011        & 0.010         \\ \hline 
\end{tabular}
\end{table}
\begin{table}[H]
\centering
\caption{MAE and MSE of the estimators for $Z^{2}$}
\begin{tabular}{ccccc}
\hline
                                      &    & $T=10,n=1000$ & $T=8,n=3000$ & $T=10,n=5000$ \\ \hline
\multirow{2}{*}{MAE$(\times 10^{3})$} & NW & 0.532         & 0.630        & 0.529         \\
                                      & LL & 0.201         & 0.182        & 0.139         \\  \hline 
\multirow{2}{*}{MSE$(\times 10^{7})$} & NW & 3.776         & 5.370        & 3.716         \\  
                                      & LL & 0.533         & 0.342        & 0.246         \\  \hline 
\end{tabular}
\end{table}
\begin{figure}[H]
\centering
\begin{minipage}[H]{0.35\textwidth}
\includegraphics[width=0.95\textwidth]{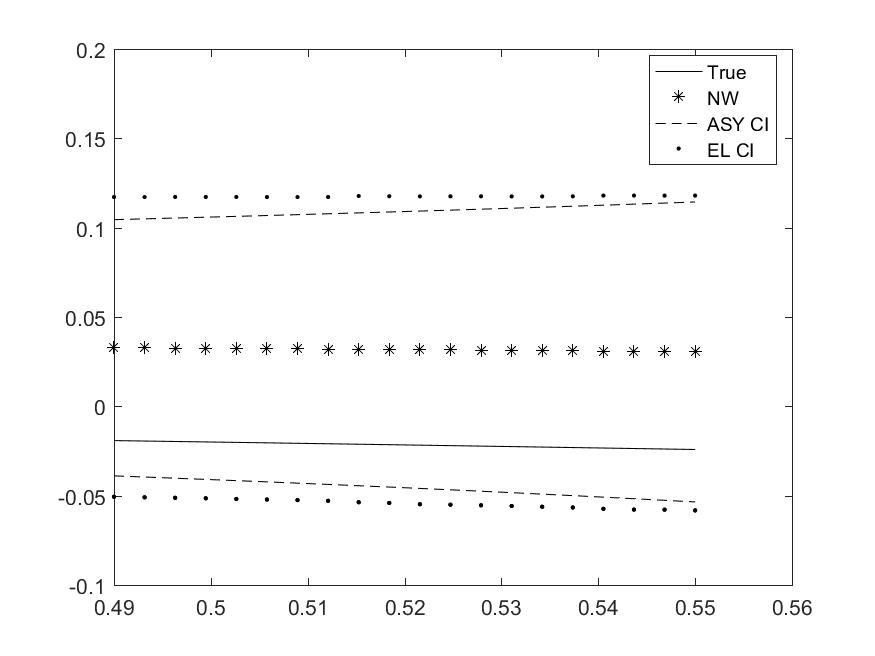}
\caption*{(a) Confidence interval of $f$}
\end{minipage}
\begin{minipage}[H]{0.35\textwidth}
\includegraphics[width=0.95\textwidth]{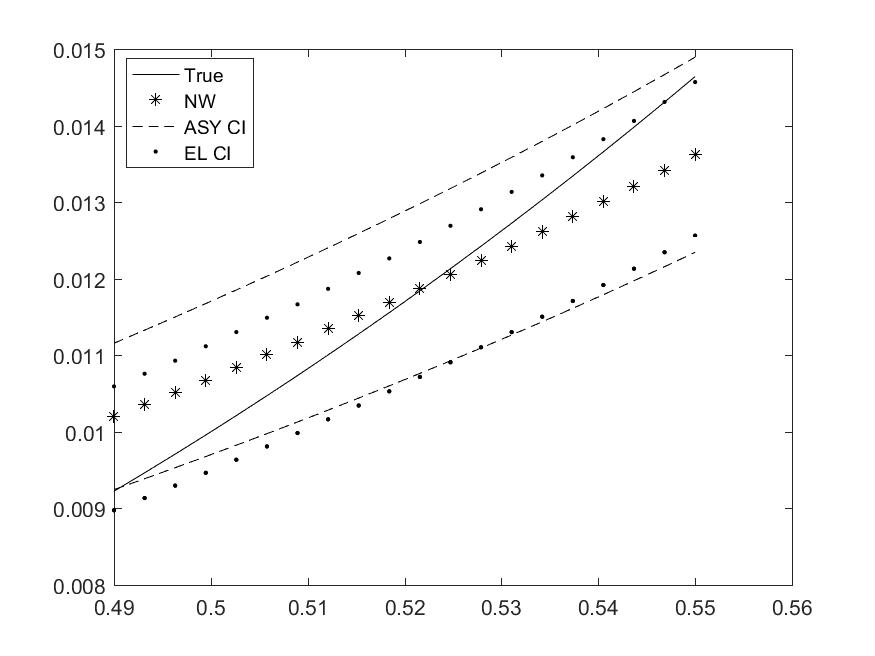}
\caption*{(b) Confidence interval of  $Z^{2}$}
\end{minipage}
\caption{The solid lines represent the true values, the '*' lines represent  the local constant estimators, the '--' dotted lines represent  the confidence intervals based on the common method using the asymptotic normality of the estimator,  the '$\cdot$' dotted lines represent the confidence intervals applying the empirical likelihood method. }
\end{figure}

The above simulations show that the estimators behave better as the time span increases and the observation interval decreases. The local linear estimators have a better performance because their biases do not depend on the first derivative of the corresponding function. Moreover, the empirical likelihood method works well at most cases.

\end{document}